\documentclass[12pt]{amsart}
\usepackage[T1]{fontenc}
\usepackage[latin1]{inputenc}
\usepackage{typearea}
\usepackage{geometry}
\usepackage{amsmath}
\usepackage{amssymb}
\usepackage{latexsym}
\usepackage{enumerate}
\usepackage{amsthm}
\usepackage[all]{xy}
\usepackage{hhline}
\usepackage{epsf} 
\usepackage{cite}

\newtheorem{theorem}{{\sc Theorem}}[section]
\newtheorem{cor}[theorem]{{\sc Corollary}}

\newtheorem{lemma}[theorem]{{\sc Lemma}}
\newtheorem{prop}[theorem]{{\sc Proposition}}

\theoremstyle{remark}
\newtheorem{remark}[theorem]{{\sc Remark}}

\theoremstyle{definition}

\newcommand{\R}{\mathbb{R} }
\newcommand{\N}{\mathbb{N} }

\newcommand{\F}{\mathcal{F}}

\newcommand{\Z}{\mathbb{Z}}

\newcommand{\calS}{\mathcal{S}}

\newcommand{\Prob}{\mathbb{P}}
\newcommand{\E}{\mathbb{E}}

\newcommand{\ga}{\gamma}

\newcommand{\Pot}{\mathcal{P}}

\newcommand{\om}{\omega}

\providecommand{\abs}[1]{\lvert #1\rvert}

\DeclareMathOperator{\Var}{Var}

\DeclareMathOperator{\Cov}{Cov}

\renewcommand{\phi}{\varphi}
\renewcommand{\epsilon}{\varepsilon}
\newcommand{\eps}{\varepsilon}
\renewcommand{\rho}{\varrho}

\begin{document}
\title{Recurrence for the frog model with drift on $\mathbb{Z}^d$}
\author{Christian D\"obler \and Lorenz Pfeifroth}
\thanks{Technische Universit\"at M\"unchen, Fakult\"at f\"ur Mathematik, Bereich M5, D-85748 M\"unchen, Germany. \\
christian.doebler@tum.de }
\begin{abstract}
In this paper we present a recurrence criterion for the frog model on $\Z^d$ with an i.i.d. initial configuration of sleeping frogs and such that the underlying random walk has a drift to the right.
\end{abstract}

\maketitle

\section{Introduction}\label{intro}
The frog model is a certain model of interacting random walks on a graph. Imagine a graph $G=(V,E)$ with a distinguished vertex $x_0\in V$, called the \textit{origin}. At time $0$, there is exactly one active frog at $x_0$ and on each vertex $x\in V\setminus\{x_0\}$ there is a number $\eta_x\in\Z_+:=\Z\cap [0,\infty)$ of sleeping frogs. 
The frog at $x_0$ now starts 
a nearest-neighbour random walk on the graph $G$. If it hits a vertex $x$ with $\eta_x>0$ sleeping frogs, they all become active at once and start performing  nearest-neighbour random walks, independently of each other and of the original frog. More generally, each time an active frogs hits a vertex $x\in V$ with $\eta_x>0$ sleeping frogs, they 
all become active at once and start nearest-neighbour random walks, independently of each other and of all other frogs. In this description, the transition function of the \textit{underlying random walk} is supposed to be the same for all frogs. The frog model is called \textit{recurrent}, if the probability that the origin $x_0$
is visited infinitely often equals $1$, otherwise the model is called \textit{transient}. The frog model with $V=\Z^d$, $E$ the set of nearest-neighbour edges on $\Z^d$, $x_0:=0$, $\eta_x=1$ for each $x\in\Z^d\setminus\{0\}$ and the underlying random walk being simple random walk (SRW) on $\Z^d$ was studied by Telcs and Wormald \cite{TeWo}.
They showed in particular that the frog model is recurrent for each dimension $d$. This result was refined by Popov \cite{Pop01}, who considered frogs in a random environment. More precisely, he considered the situation, where there is, for each $x\in\Z^d\setminus\{0\}$, originally one sleeping frog at $x$ with probability 
$p(x)$ and no frog with probability $1-p(x)$, independently of all other vertices, and found the exact rate of decay for the function $p(x)$ to distinguish transience from recurrence. Another modification of the model is to consider the frog model with death, allowing activated particles to disappear after a random, e.g. geometric, lifetime. 
Such a model, also with a random initial configuration of sleeping frogs, was analyzed by \cite{AMP02} who proved phase transition results for both survival and recurrence of the particle system using a slightly different definition of recurrence. Note that the frog model on $\Z^d$ (without death and) with SRW is trivially recurrent for $d=1,2$, due to P\'{o}lya's theorem. Thus, in \cite{GanSch} Gantert and Schmidt considered 
the frog model on $\Z$ with the underlying random walk having a drift to the right. They considered both fixed and i.i.d. random initial configurations $(\eta_x)_{x\in\Z\setminus\{0\}}$ of sleeping frogs and derived precise criteria to separate transience from recurrence. In the case of an i.i.d. initial configuration of sleeping frogs they 
also proved a $0-1$ law, which says that the probability of infinitely many returns to $0$ equals $1$, if $E[\log^+(\eta_1)]=\infty$, and equals $0$, otherwise, independently of the concrete value of the drift. The purpose of the present note is to prove that the frog model on $\Z^d$, $d\geq2$, with an i.i.d. initial configuration of sleeping frogs is recurrent, whenever the distribution giving the number of sleeping frogs per site is heavy-tailed enough. 
The paper is structured as follows: In Section \ref{mainthm} we give a precise description of the model we consider and state our main theorem, Theorem \ref{mt}. In Section \ref{proofmt} we give the proof of Theorem \ref{mt} and finally, in Section \ref{lemmas} we give proofs of two auxiliary lemmas, which we need in Section \ref{proofmt} 
in order to prove Theorem \ref{mt}.

\section*{Acknowledgements}
We would like to thank Silke Rolles and Nina Gantert for useful discussion und comments.

\section{Setting and main theorem}\label{mainthm}

As mentioned above, we consider recurrence of the frog model on $\Z^d$ with an i.i.d. initial configuration and such that the underlying random walk has a drift to the right. We denote by $\calS$ the set of all possible initial configurations of sleeping frogs, i.e. 
\begin{equation*}
\calS:=\bigl\{\eta=(\eta_x)_{x\in\Z^d\setminus\{0\}}\in\Z_+^{\Z^d\setminus\{0\}}\bigr\}\,.
\end{equation*}
Further, we denote by $p$ the transition function of the underlying nearest-neighbour random walk. Thus, letting $\mathcal{E}:=\{\pm e_j\,:\,1\leq j\leq d\}$, where $e_j$ denotes the $j$-th standard basis vector in $\R^d$, $j=1,\dotsc,d$, we assume that $p:\Z^d\rightarrow[0,\infty)$ is a function such that
\begin{equation*}
\sum_{e\in\mathcal{E}}p(e)=1
\end{equation*}
and $p(x)=0$ for all $x\in\Z^d\setminus\mathcal{E}$. In order to make the random walk irreducible, we will further assume that $0<p(e)<1$ holds for all $e\in\mathcal{E}$. Additionally, we will abuse notation to write $p(x,y):=p(y-x)$ also for the corresponding transition matrix. Since we assume that the underlying random walk has a drift to the right, we suppose that there is an $a\in(0,1)$ such that 
\begin{equation}\label{drift}
m:=\sum_{e\in\mathcal{E}}p(e)e=ae_1\,.
\end{equation}
Since the transition function $p$ will be kept fixed throughout, we omit it from the notation.
For a fixed $\eta\in\calS$ we denote by $P_\eta$ a probability measure on a suitable measurable space $(\Omega,\F)$, which describes the evolution of the frog model with initial configuration $\eta$ and underlying random walk given by the transition function $p$ as described in the introduction. We refrain from giving a mathematical construction of the frog model with respect to $\eta$ but refer the interested reader to \cite{Pop03}. Now, let $\mu$ be a probability distribution on $(\Z_+,\Pot(\Z_+))$ and let $\Prob_\mu$ be the corresponding product measure on $(\Z_+^{\Z^d\setminus\{0\}},\Pot(\Z_+)^{\otimes\Z^d\setminus\{0\}})$, i.e. 
$\Prob_\mu=\mu^{\otimes\Z^d\setminus\{0\}}$. The corresponding expectation operator will be denoted by $\E_\mu$. Finally, we denote by 
$P$ the $\Prob_\mu$-mixture of the measures $P_\eta$, i.e.
\begin{equation}\label{defanneal}
P(A)=\int_{\Z_+^{\Z^d\setminus\{0\}}} P_\eta(A) \Prob_\mu(d\eta)\,,\quad A\in\F\,.
\end{equation}
Thus, the measure $P$ describes the evolution of the frog model with respect to a random i.i.d. initial configuration 
$\eta$. From \eqref{defanneal} we can make the following easy but important observation: \\
\emph{An event $A\in\F$ holds $P$-a.s. if and only if it holds $P_\eta$-a.s. for $\Prob_\mu$-a.a. $\eta\in\calS$.}

With this notation at hand, we are ready to state the main result of this note:
\begin{theorem}\label{mt}
If, additionally to the above assumptions, the distribution $\mu$ is such that 
$\E_\mu\bigl[\log^+(\eta_x)^{\frac{d+1}{2}}\bigr]=\sum_{j=2}^\infty \log(j)^{\frac{d+1}{2}}\mu(j)=\infty$,
then the frog model with drift to the right and i.i.d. initial configuration $\eta\sim\Prob_\mu$ is recurrent, i.e.
\[P\bigl(0\text{ is visited infinitely often }\bigr)=1\,.\]
\end{theorem}

\begin{remark}\label{mtrem}
\begin{enumerate}[(a)]
 \item If $d=1$, Theorem \ref{mt} reduces to one of the results by Gantert and Schmidt \cite{GanSch} that the frog model is recurrent, if $\E_\mu[\log^+(\eta_1)]=+\infty$. 
 \item Thanks to discussion with Serguei Popov we believe that for the frog model on $\Z^d$, $d\geq2$, with an i.i.d. initial configuration of sleeping frogs, in general, the question of transience and recurrence depends on the concrete value of the drift, unless the distribution of $\eta$ is heavy-tailed enough, as in the situation 
 of Theorem \ref{mt}. Establishing a phase transition result for recurrence and transience for distributions of $\eta$ with lighter tails is part of a follow-up project. 
\end{enumerate}
\end{remark}

\section{Proof of Theorem \ref{mt}}\label{proofmt}
First, we need to fix some more notation. Fix an integer $\alpha>1$, which is further specified later on and for $n\in\N=\{1,2,\ldots\}$ let 
\begin{align}\label{Fn}
F_n&:=\Bigl\{x\in\Z^d:\frac{3}{2}\alpha^{2n}\leq x_1<\alpha^{2n+2} \text{ and } |x_j|\leq\alpha^{n}\text{ for } j=2,\dotsc,d\Bigr\}\,.
\end{align}
 Furthermore, for $x,y\in\Z^d$ we denote by $f(x,y)$ the probability that the underlying random walk ever hits $y$, if it starts at $x$. Thus, if we denote this random walk by $(X_n)_{n\geq0}$, then 
$f(x,y)=P(\exists\,n\geq0:\,X_n=y|X_0=x)$. If we choose, according to our assumptions, $\eps>0$ such that $\eps\leq p(\pm e)\leq 1-\eps$ holds for each $e\in\mathcal{E}$, then we have the following lower bound for the probabilities $f(x,y)$: 
\begin{equation}\label{boundfxy}
 f(x,y)\geq  \eps^{d\abs{y-x}}\,\quad\text{for all } x,y\in\Z^d\,,
\end{equation}
where we denote by $\abs{x}:=\max_{1\leq j\leq d} \abs{x_j}$ the maximum norm of a vector $x=(x_1,\dotsc,x_d)\in\R^d$.
This follows from the fact that one can get from $x$ to $y$ in at most $d\abs{y-x}$ steps. If $y$ lies to the right of $x$, then one can do better. More precisely, we have the following bound. 
\begin{lemma}\label{htlemma}
For each finite constant $\gamma>0$, there exists a constant $c_1=c_1(\gamma,p)>0$ such that for all $x,y\in\Z^d$ with $y_1>x_1$ and $|y_j-x_j|\leq\gamma\sqrt{y_1-x_1}$, $j=2,\dotsc,d$, we have 
\[f(x,y)\geq\frac{c_1}{(y_1-x_1)^{\frac{d-1}{2}}}\,.\] 
\end{lemma}
A proof of Lemma \ref{htlemma} is given in Section \ref{lemmas}.
The following lemma about the behaviour of maxima of nonnegative i.i.d. random variables is one of the cornerstones of the proof of Theorem \ref{mt}. Throughout, we denote by $\abs{A}$ the cardinality of the set $A$.

\begin{lemma}\label{ml}
Let $r>0$ be a finite constant, $J$ be a countably infinite index set and let $(Y_j)_{j\in J}$ be a sequence of nonnegative i.i.d. random variables such that $E[\log^+(Y_j)^r]=\infty$. Furthermore, let $(L_i)_{i\in\N}$ be a sequence of pairwise disjoint subsets of $J$ such that $l_i:=\abs{L_i}\geq c_2\beta^{c_3i}$ holds for each $i\in\N$, where 
$c_2,c_3>0$ and $\beta>1$ are constants (Here, $\beta$ needs not necessarily be an integer). For $i\in\N$ define 
\begin{equation}\label{defMi}
M_i:=\max_{j\in L_i} Y_j\,.
\end{equation}
Then, for each finite constant $c>0$ it holds that 
\begin{equation}\label{statml}
P\bigl(M_i\geq\exp\bigl(c\beta^{\frac{c_3i}{r}}\bigr)\text{ for infinitely many }i\in\N\bigr)=1\,.
\end{equation}
\end{lemma}

The proof of Lemma \ref{ml} is given in Section \ref{lemmas}.\\
Now we can proceed to the proof of Theorem \ref{mt}, which uses a technique from \cite{Pop01}. 
Choose the positive integer $\alpha$ such that
\begin{align}\label{alpha}
\alpha\geq\max\Bigl(3,\frac{1}{c_1}\Bigr),
\end{align}
where $c_1$ is the constant from Lemma \ref{htlemma}. 
Further, we define 
\begin{equation}\label{defVn}
V_n:=\{x\in\Z^d\,:\,\abs{x}\leq\alpha^{2n}\},\quad n\in\N\,.
\end{equation}
 Let us repeat the following important observation from \cite{Pop01}: 
For recurrence of the frog model, everything that matters is the \textit{trajectories} of the activated frogs. The actual moment that a certain frog gets activated is unimportant. Thus, if we know that a certain frog starting from vertex $x\in\Z^d$ will sooner or later be at vertex $y$, we will say that the frogs at vertex $y$ are activated by a frog from $x$, even if it is not the first frog to visit vertex $y$. We will call a vertex $x\in\Z^d$ \textit{active} if at least one active frog ever visits $x$.  \\
Fix $k\in\N$ with $k\geq 2$ and define the event 
\begin{align*}\label{defAk}
A_k &:=\bigl\{\text{at a certain moment and at some vertex $x_k\in V_k\setminus V_{k-1}$ at least}\\ 
&\,\alpha^{(d+1)(k-1)} \text{ frogs get activated by the initial frog starting from the origin}\bigr\}\notag\,.
\end{align*}
In the following, we will implicitly be conditioning on the event $A_k$. Note that the event $A_k$ only depends on the randomness coming from the path of the initial frog and from the values of the $\eta_x$, where $x\in V_k$. Define 
\begin{equation}\label{defB0D0}
B_0:=\{x_k\}\,\quad D_0:=\emptyset\,.
\end{equation}
We will try to construct inductively sets $D_i\subseteq F_{k+i-1}$, $i\in\N$, such that with 
\begin{equation*}
B_i=F_{k+i-1}\setminus D_i
\end{equation*}
the following hold: We have 
\begin{equation}\label{cardBiDi}
\abs{D_i}=\alpha^{(d+1)(i+k-1)}\quad\text{and}\quad \abs{B_i}\geq\alpha^{(d+1)(i+k-1)}
\end{equation}
and all the sites in $D_i$ are visited by frogs starting from $B_{i-1}$, $i\in\N$.
Furthermore, denoting for each $i\in\N$ and $y\in F_{k+i-1}$ by $\zeta_y$ the indicator of the following event
\[\{\text{at least one active frog starting from }B_{i-1}\text{ eventually visits }y\},\]
we require that
\begin{equation}\label{Bievent}
\sum_{y\in B_i}\zeta_y\eta_y\geq\alpha^{(d+1)(i+k-1)}
\end{equation}
holds for each $i\in\N$.
Note that by the definition of the sets $F_n$ in \eqref{Fn} we have 
\begin{equation}\label{cardFn}
\abs{F_n}=\alpha^{2n}\bigl(\alpha^2-\frac{3}{2}\bigr)\bigl(2\alpha^n+1\bigr)^{d-1}
\end{equation}
 and hence, since $\alpha^2\geq4$, we get 
\begin{equation}\label{lbcardFn}
\abs{F_n}\geq \frac{5}{2}2^{d-1}\alpha^{n(d+1)}
\end{equation}
 and 
\begin{equation}\label{ubcardFn}
\abs{F_n}\leq 3^{d-1}\alpha^2\alpha^{n(d+1)}\leq \alpha^{d+1} \alpha^{n(d+1)}\,.
\end{equation}
Note that by \eqref{lbcardFn} for all $i\in\N$
\begin{align*}
\abs{F_{k+i-1}}-2\alpha^{(d+1)(i+k-1)}
\geq\alpha^{(d+1)(i+k-1)}\bigl(\frac{5}{2}2^{d-1}-2\bigr)>0\,.
\end{align*}
Thus, in principle, there are enough vertices in $F_{k+i-1}$ to form disjoint sets $B_i$ and $D_i$ as required. The next thing to do is prove that, in fact, with high enough probability enough vertices in $F_{k+i-1}$ are visited by frogs starting from $B_{i-1}$ and also that the number of activated frogs is large enough for \eqref{Bievent} to occur. Suppose that for $0\leq j\leq i$ the sets $B_j$ and $D_j$ have already been succesfully constructed. We will soon be more precise about what this exactly means. 
For $i\in\Z_+$ we define events $G_{i,1}$, $G_{i,2}$ and $G_i$ as follows: Let 
\begin{equation}\label{defG1}
G_{i,1}:=G_{i,1}^{(k)}:=\left\{\sum_{y\in F_{k+i}}\zeta_y\geq2\alpha^{(d+1)(i+k)}\right\}\,.
\end{equation}
If $G_{i,1}$ happens than we can construct the set $D_{i+1}$ by choosing exactly $\alpha^{(d+1)(i+k)}$ vertices from 
$F_{k+i}$ that are visited by frogs starting from $B_i$ according to \eqref{defG1} and let $B_{i+1}:=F_{k+i}\setminus D_{i+1}$. Then, we define 
\begin{equation}\label{defG2}
G_{i,2}:=G_{i,2}^{(k)}:=\left\{\sum_{y\in B_{i+1}}\zeta_y\eta_y\geq\alpha^{(d+1)(i+k)}\right\}\quad\text{and}
\quad G_i:=G_{i}^{(k)}:=G_{i,1}\cap G_{i,2}\,.
\end{equation}
We will call the $i$th inductive step \textit{succesful} if $G_i$ happens (given that\\ $A_k,G_0,\dotsc,G_{i-1}$ happen). As just explained, in this case it is possible to form subsets $B_{i+1},D_{i+1}$ of $F_{k+i}$ with all the desired properties.
In what follows we will implicitly be conditioning on the event $A_k\cap G_0\cap\ldots\cap G_{i-1}$ but will suppress this from the formulas for ease of notation. Also, for the computations which follow the following remark from \cite{Pop01} will be crucial:
Suppose that there are disjoint subsets $A,B\subseteq\Z^d$ and we know that for each $x\in A$ there is a frog starting from a vertex $y\in B$ which activates the frogs at vertex $x$. Then, all the frogs starting from $A$ are independent, since we only allow for interaction when an active frog is waking up a sleeping frog.\\
Note that for all $i\in\Z_+$ and all $x\in F_{k+i-1}, y\in F_{k+i}$ we have 
\begin{equation}\label{boundfc}
\frac{1}{2}\alpha^{2(k+i)}\leq(y_1-x_1)\leq\alpha^{2(k+i+1)}\,.
\end{equation} 
\begin{lemma}\label{zetalemma}
Under the above assumptions and conditionally on the event $A_k\cap G_0\cap\ldots\cap G_{i-1}$, we have for all $i\in\Z_+$ and all $y,z\in F_{k+i}$:
\begin{align}
E[\zeta_y]&\geq 1-\exp(-2)\label{Ezeta}\\
\Var(\zeta_y)&\leq1\label{Varzeta}\\
\Cov(\zeta_y,\zeta_z)&\leq\exp\bigl(-\alpha^{k+2(i-1)}\bigr)\leq\exp\bigl(-i\alpha^{k-2}\bigr)\label{Covzeta}
\end{align}
\end{lemma}

\begin{proof}[Proof of Lemma \ref{zetalemma}]
By the above remark we have 
\begin{equation}\label{zl1}
E[\zeta_y]=P(\zeta_y=1)=1-P(\zeta_y=0)= 1-\prod_{x\in B_i}\bigl(1-f(x,y)\bigr)^{\eta_x}\,.
\end{equation}
Now,  from Lemma \ref{htlemma}, \eqref{boundfc} and the fact that \eqref{Bievent} holds since we are conditioning on $G_{i-1}$, we obtain
\begin{align}\label{zl2}
\prod_{x\in B_i}\bigl(1-f(x,y)\bigr)^{\eta_x}&\leq\prod_{x\in B_i}\left(1-\frac{c_1}{(y_1-x_1)^{\frac{d-1}{2}}}\right)^{\eta_x}\nonumber\\
&\leq\left(1-c_1\alpha^{-\frac{2(k+i+1)(d-1)}{2}}\right)^{\alpha^{(d+1)(k+i-1)}}\notag\\
&=\left(1-c_1\alpha^{-(k+i+1)(d-1)}\right)^{\alpha^{(d+1)(k+i-1)}}\,.
\end{align}
By the inequality 
\begin{equation}\label{xyineq}
(1-x)^y\leq\exp(-xy)
\end{equation}
valid for all $x\in(0,1)$ and $y>0$, we have 
\begin{align}\label{zl3}
\left(1-c_1\alpha^{-(k+i+1)(d-1)}\right)^{\alpha^{(d+1)(k+i-1)}}
&\leq\exp\left(-c_1\frac{\alpha^{(d+1)(k+i-1)}}{\alpha^{(d-1)(k+i-1)}}\right)\notag\\
&\leq\exp\bigl(-c_1\alpha^{2(k+i-1)}\bigr)\,.
\end{align}
Now, using $k\geq 2$, $i\geq0$ and $\alpha\geq1/c_1$ we conclude from \eqref{zl1}, \eqref{zl2} and \eqref{zl3} that 
\[E[\zeta_y]\geq 1-\exp(-2)\,,\]
proving \eqref{Ezeta}. Since $0\leq\zeta_y\leq1$ \eqref{Varzeta} is trivially true. To prove \eqref{Covzeta}, note that 
\begin{align}\label{zl4}
\Cov(\zeta_y,\zeta_z)&=\Cov(1-\zeta_y,1-\zeta_z)=P(\zeta_y=\zeta_z=0)-P(\zeta_y=0)P(\zeta_z=0)\notag\\
&\leq P(\zeta_y=0)\leq \exp\bigl(-c_1\alpha^{2(k+i-1)}\bigr)
\end{align}
from \eqref{zl3}. Using $\alpha^k\geq\alpha\geq1/c_1$ and $\alpha^{2i}\geq i$ we obtain \eqref{Covzeta}.\\
\end{proof}

The next lemma gives an upper bound on the probability that the event $G_{i,1}$ does not happen (conditionally on the event 
$A_k\cap G_0\cap\ldots\cap G_{i-1}$). 

\begin{lemma}\label{Gicomp}
There is a finite constant $c_4=c_4(\alpha,d)>0$, which is independent of $k$, such that for all $i\in\N$
\begin{equation*}\label{ubGicomp}
P(G_{i,1}^c)=P\left(\sum_{y\in F_{k+i}}\zeta_y<2\alpha^{(d+1)(i+k)}\right)
\leq c_4\Bigl(\alpha^{-(k+i)(d+1)}+\exp\bigl(-i\alpha^{k-2}\bigr)\Bigr)
\end{equation*}
and
\begin{equation}\label{ubG0comp}
P(G_{0,1}^c)=P\left(\sum_{y\in F_{k}}\zeta_y<2\alpha^{(d+1)k}\right)
\leq c_4\Bigl(\alpha^{-k(d+1)}+\exp\bigl(-\alpha^{k-2}\bigr)\Bigr)\,.
\end{equation}
\end{lemma}

\begin{proof}[Proof of Lemma \ref{Gicomp}]
By inequalities \eqref{lbcardFn} and \eqref{Ezeta} we have 
\begin{equation}\label{Gi1}
\sum_{y\in F_{k+i}}E[\zeta_y]\geq\abs{F_{k+i}}(1-\exp(-2))\geq \frac{5}{2}2^{d-1}\alpha^{(k+i)(d+1)}(1-\exp(-2))\,.
\end{equation}
Thus, using the simple inequality $P(X\leq a)\leq P(X\leq b)$ if $a<b$ we obtain 
\begin{align}\label{Gi2}
&P\left(\sum_{y\in F_{k+i}}\zeta_y<2\alpha^{(d+1)(i+k)}\right)\notag\\
&=P\left(\sum_{y\in F_{k+i}}\bigl(\zeta_y-E[\zeta_y]\bigr)<2\alpha^{(d+1)(i+k)}-\sum_{y\in F_{k+i}}E[\zeta_y]\right)\notag\\
&\leq P\left(\sum_{y\in F_{k+i}}\bigl(\zeta_y-E[\zeta_y]\bigr)<-\alpha^{(d+1)(i+k)}\Bigl(\frac{5}{2}2^{d-1}(1-\exp(-2))-2\Bigr)\right)
\end{align}
Now note that we have 
\begin{align}\label{Gi3}
\frac{5}{2}2^{d-1}(1-\exp(-2))-2\geq \frac{5}{2}(1-\exp(-2))-2=:c>0
\end{align}
for all $d\geq1$. Note that $c$ does not depend on $k$. Hence, by \eqref{Gi3}, Chebyshev's inequality, inequalities \eqref{ubcardFn}, \eqref{Varzeta} and the second inequality in \eqref{Covzeta} we have for each $i\geq1$. 
\begin{align}\label{Gi4}
P(G_{i,1}^c)&\leq c^{-2}\alpha^{-2(d+1)(i+k)}\left(\sum_{y\in F_{k+i}}\Var(\zeta_y)+\sum_{\substack{y,z\in F_{k+i}:\\y\not=z}}
\Cov(\zeta_y,\zeta_z)\right)\notag\\
&\leq c^{-2}\alpha^{-2(d+1)(i+k)}\Bigl(\alpha^d \alpha^{(k+i)(d+1)}+\alpha^{2d} \alpha^{2(k+i)(d+1)}
\exp\bigl(-i\alpha^{k-2}\bigr)\Bigr)\notag\\
&\leq c_4\Bigl(\alpha^{-(k+i)(d+1)}+\exp\bigl(-i\alpha^{k-2}\bigr)\Bigr)\,,
\end{align}
where $c_4=c^{-2}\alpha^{2d}$ is also independent of $k$. For $i=0$ we obtain the desired upper bound \eqref{ubG0comp} by using the first inequality in \eqref{Covzeta} instead of the second one.\\
\end{proof}

Next, we aim at bounding below the conditional probability of $G_{i,2}$ given that $G_{i,1}$ happens. Note that if 
$G_{i,1}$ happens, the set $B_{i+1}$ is well-defined and also we have 
\begin{equation}\label{Gi21}
P(G_{i,2}|G_{i,1})\geq P\Bigl(\sum_{j=1}^{a_i} Y_j\geq a_i\Bigr)\,,
\end{equation}
where $Y_1,Y_2,\dotsc$ are i.i.d. with the same distribution $\mu$ as the $\eta_x$ and we write $a_i:=\alpha^{(d+1)(k+i)}$, $i\in\N$, for short. This follows directly from independence and \eqref{cardBiDi}. Since the $Y_j$ are nonnegative and have infinite mean, we know from Cram\'{e}r's theorem (see Theorem 2.2.3 and the following Remark (c)in \cite{DZ}) 
that with the notation $S_n:=\sum_{j=1}^n Y_j$, $n\in\N$, we have
\begin{equation}\label{cramer}
P(S_n\leq n)\leq 2\exp\bigl(-n b\bigr)\,, n\in\N\,,
\end{equation}
where $b=I(1)>0$ is the value at $1$ of the Legendre-Fenchel transform $I(x)$ of the cumulant generating function of 
$Y_1$. That $I(1)>0$ also follows from the fact that $Y_1$ is nonnegative and has infinite mean. 
From \eqref{Gi21} and \eqref{cramer} we conclude that for each $i\geq0$
\begin{equation}\label{Gi22}
P(G_{i,2}|G_{i,1})\geq P(S_{a_i}\geq a_i)\geq 1-P(S_{a_i}\leq a_i)\geq 1-2\exp\bigl(-ba_i)\bigr)\,,
\end{equation}
where we let $b:=I(1)>0$. Now, using 
\begin{align*}
P(G_i^c)&=1-P(G_i)=1-P(G_{i,2}|G_{i,1})P(G_{i,1})=1-P(G_{i,2}|G_{i,1})\bigl(1-P(G_{i,1}^c)\bigr)\\
&\leq 1-P(G_{i,2}|G_{i,1})+P(G_{i,1}^c)
\end{align*}
and $a_i\geq i\alpha^k$, from Lemma \ref{Gicomp} and \eqref{Gi22} we immediately infer the following lemma.
\begin{lemma}\label{Gic}
With the constant $c_4=c_4(\alpha,d)$ from Lemma \ref{Gicomp} we have
\item \begin{equation}\label{ubGic}
P(G_{i}^c)\leq c_4\Bigl(\alpha^{-(k+i)(d+1)}+\exp\bigl(-i\alpha^{k-2}\bigr)\Bigr)+2\exp\bigl(-ib\alpha^k\bigr)\,,i\in\N\,,
\end{equation}
and
\begin{equation}\label{ubG0c}
P(G_{0}^c)\leq c_4\Bigl(\alpha^{-k(d+1)}+\exp\bigl(-\alpha^{k-2}\bigr)\Bigr)
+2\exp\bigl(-b\alpha^k\bigr)\,.
\end{equation}
\end{lemma}

Now, for $x\geq0$, define the function 
\begin{align}\label{defg}
g(x)&:=c_4\left(\frac{\alpha^{-x(d+1)}}{1-\alpha^{-(d+1)}}+\frac{\exp\bigl(-\alpha^{x-2}\bigr)}{1-\exp\bigl(-\alpha^{x-2}\bigr)}+\exp\bigl(-\alpha^{x-2}\bigr)\right)\\
&\,+2\left(\exp\bigl(-b\alpha^x\bigr)+\frac{\exp\bigl(-b\alpha^x\bigr)}{1-\exp\bigl(-b\alpha^x\bigr)}\right)
\end{align}
 and note that 
\begin{equation}\label{asymg}
\lim_{x\to\infty} g(x)=0\,.
\end{equation}
From Lemma \ref{Gicomp} and the multiplication rule for conditional probabilites, we obtain that under our initial assumption that the event $A_k$ happens we have 
\begin{align}\label{problb}
P\Bigl(\bigcap_{i=0}^\infty G_i\Bigr)&=\lim_{m\to\infty}P\Bigl(\bigcap_{i=0}^m G_i\Bigr)
=\lim_{m\to\infty}\prod_{i=0}^m\bigl(1-P\bigl(G_i^c|G_0\cap\ldots \cap G_{i-1}\bigr)\bigr)\notag\\
&\geq\lim_{m\to\infty}\bigl(1-\sum_{i=0}^m P\bigl(G_i^c|G_0\cap\ldots \cap G_{i-1}\bigr)\bigr)\notag\\
&=1-\sum_{i=0}^\infty P\bigl(G_i^c|G_0\cap\ldots \cap G_{i-1}\bigr)
\geq 1-g(k)\,,
\end{align}
where we have used the simple inequality 
\[\prod_{i=0}^m(1-p_i)\geq 1-\sum_{i=0}^m p_i\]
valid for numbers $p_0,\dotsc,p_m\in[0,1]$.

\begin{prop}\label{condprop}
Fix $k\in\N$. Assume for the frog model that the i.i.d. random variables $\eta_x$, $x\in\Z^d\setminus\{0\}$ satisfy $\E_\mu[\log^+(\eta_x)^{\frac{d+1}{2}}]=\infty$. Then, if the event $A_k$ happens and, thus, $B_0$ can be constructed as in \eqref{defB0D0}, we have 
\begin{equation*}
P\Bigl(0\text{ is visited infinitely often }\bigl|\,\bigcap_{i=0}^\infty G_i\Bigr)=1\,.
\end{equation*}
\end{prop}

\begin{proof}[Proof of Proposition \ref{condprop}]
First note that, if $k\geq1$ is fixed, the sets $D_i$, $i\in\N$, satisfy $D_i\subseteq F_{k+i-1}$ and, hence, we have 
$D_i\cap V_k=\emptyset$ and also $D_i\cap\bigcup_{j\in\Z_+} B_j=\emptyset$ for each $i\in\N$. The event $A_k$ does not depend on the values of the random variables $\eta_x$ for $x\notin V_k$. Furthermore, the event $\bigcap_{j\in\Z_+} G_j$ only depends on the $\eta_x$ such that $x\in A_k\cup\bigcup_{j\in\Z_+} B_j$. Thus, after conditioning on $A_k$ and on $\bigcap_{j\in\Z_+} G_j$, by independence, we still have the i.i.d. property for the $\eta_x$, where $x\in\bigcup_{i\in\Z_+} D_i$.
This will allow us to apply Lemma \ref{ml} below.
Note that for each fixed configuration $\eta_x$, $x\in\bigcup_{i\in\Z_+} D_i$, by \eqref{boundfxy} we have 
\begin{align}\label{cp1}
\sum_{i=1}^\infty\sum_{x\in D_i}\eta_x f(x,0)&\geq \sum_{i=1}^\infty \sum_{x\in D_i}\eta_x\eps^{d\abs{x}}
\geq\sum_{i=1}^\infty \eps^{d\alpha^{2k+2i}}\sum_{x\in D_i}\eta_x\notag\\
&\geq \sum_{i=1}^\infty \delta^{\alpha^{2i}}M_i\,,
\end{align}
where $\delta:=\eps^{d\alpha^{2k}}\in(0,1)$ and $M_i:=\max_{x\in D_i}\eta_x$, $i\in\N$. 
For $i\in\N$ let $l_i:=\abs{D_i}=\alpha^{(k-1)(d+1)}\alpha^{i(d+1)}$. Then, by using Lemma \ref{ml} 
with $L_i:=D_i$, $c:=-\log\delta$, $c_2=\alpha^{(k-1)(d+1)}$, $c_3=d+1$, $r=\frac{d+1}{2}$ and $\beta=\alpha$  we obtain that $\Prob_\mu$-a.s. 
\begin{equation}\label{cp2}
M_i\geq\exp\bigl(c\alpha^{2i}\bigr)\text{  for infinitely many }i\in\N\,.
\end{equation}
Hence, $\Prob_\mu$-a.s., there is a strictly increasing sequence $(i_m)_{m\in\N}$ of positive integers such that for all $m\in\N$
\begin{equation}\label{cp3}
M_{i_m}\geq\exp\bigl(c\alpha^{2i_m}\bigr)\,.
\end{equation}
Thus, from \eqref{cp1} and \eqref{cp3} we have $\Prob_\mu$-a.s.
\begin{align}\label{cp4}
\sum_{i=1}^\infty\sum_{x\in D_i}\eta_x f(x,0)&\geq \sum_{m=1}^\infty \delta^{\alpha^{2i_m}}M_{i_m}
\geq \sum_{m=1}^\infty \delta^{\alpha^{2i_m}}\exp\bigl(c\alpha^{2i_m}\bigr)\notag\\
&=\sum_{m=1}^\infty 1=\infty\,.
\end{align}
By construction, for each $i\in\N$, the frogs in $D_i$ get activated by frogs starting from $B_{i-1}$. Hence, by the remark before Lemma \ref{zetalemma}, all frogs in $\bigcup_{i=1}^\infty D_i$ are independent. Hence, from \eqref{cp4} and the second Borel-Cantelli lemma we conclude that $\Prob_\mu$-a.s.
\begin{equation*}
P_\eta\Bigl(0\text{ is visited infinitely often }\bigl|\,\bigcap_{i=0}^\infty G_i\Bigr)=1\,.
\end{equation*}
Thus, also 
\begin{equation*}
P\Bigl(0\text{ is visited infinitely often }\bigl|\,\bigcap_{i=0}^\infty G_i\Bigr)=1\,,
\end{equation*}
as claimed.\\
\end{proof}

Now, note that from \eqref{problb} and Proposition \ref{condprop} we have 
\begin{align}\label{keyineq}
P\Bigl(0\text{ is visited infinitely often }\Bigr)&\geq P\Bigl(0\text{ is visited infinitely often }\bigl|\,\bigcap_{i=0}^\infty G_i\Bigr) P\Bigl(\bigcap_{i=0}^\infty G_i\Bigr)\notag\\
&\geq 1-g(k)\,.
\end{align}
Since $\lim_{k\to\infty}g(k)=0$ by \eqref{keyineq} the proof of Theorem \ref{mt} will be completed, if we can show that 
$P$-a.s. the event $A_k$ happens for arbitrarily large $k\in\N$. This is guaranteed by the following lemma.

\begin{lemma}\label{finlemma}
We have 
\[P\left(\limsup_{k\to\infty} A_k\right)=1\,.\]
\end{lemma}

\begin{proof}[Proof of Lemma \ref{finlemma}]
Denote by $\pi$ the path of the initial frog starting from the origin. By the properties of the underlying random walk, clearly, $\pi$ contains infinitely many different vertices. We are going to use Lemma \ref{ml} with $J=\pi$, $Y_x=\eta_x$, 
$x\in\pi$, and $r=(d+1)/2$. The pairwise disjoint sets $L_i$, $i\in\N$, are constructed inductively as follows: Let $L_1$ contain the first $\alpha^2-1$ pairwise different vertices in $\pi\setminus\{0\}$. Clearly, $L_1\subseteq V_1$. If $L_{i-1}$ for $i\geq2$ has already been constructed, let $L_i$ contain exactly the next $\alpha^{2i}-\alpha^{2i-2}$ vertices in $\pi$, which are not contained in $V_{i-1}$. Then, $L_i\subseteq V_i\setminus V_{i-1}$. Note that the sets $L_i$ satisfy 
$l_i:=\abs{L_i}\geq c_2\alpha^{2i}$, where $c_2=1-\alpha^{-2}$. Hence, from Lemma \ref{ml} (with $c_3=2$, $c=1$, $\beta=\alpha$ and $r=(d+1)/2$) we conclude that $\Prob_\mu$-a.s.
\begin{equation}\label{fl1}
M_i=\max_{x\in L_i}\eta_x\geq \exp\bigl(\alpha^\frac{4i}{d+1}\bigr)\text{  infinitely often.}
\end{equation}
In particular, $\Prob_\mu$-a.s. for each $k_0\in\N$ there exists a $k\geq k_0$ such that 
\[M_k\geq \alpha^{(k-1)(d+1)}\,,\]
implying that $P$-a.s. the event $A_k$ happens for arbitrarily large values of $k$.\\
\end{proof}

\section{Proofs of auxiliary lemmas}\label{lemmas}
This section is devoted to the proofs of Lemmas \ref{ml} and \ref{htlemma}. In order to prove Lemma \ref{ml} 
we need some facts about the behaviour of the maxima of nonnegative i.i.d. random variables, some of which rely on the following simple lemma on real sequences:

\begin{lemma}\label{rnlemma}
Let $u:[0,\infty)\rightarrow[0,\infty)$ be an increasing and invertible function and let $(y_n)_{n\in\N}$ be a sequence of numbers in the interval $[a,\infty)$. 
For $n\in\N$ let $m_n:=\max_{1\leq j\leq n} y_j$. Then, the following two conditions are equivalent:

\begin{enumerate}[{\normalfont (i)}]
 \item $m_n\geq u^{-1}(n)$ for infinitely many $n\in\N$
 \item $y_n\geq u^{-1}(n)$ for infinitely many $n\in\N$
\end{enumerate}
\end{lemma}

\begin{proof}[Proof of Lemma \ref{rnlemma}]
 Of course, (ii) trivially implies (i). So let us prove the converse. Let 
\[n_0:=\inf\{n\in\N\,:\, m_n\geq u^{-1}(n)\}\,.\]
By (i) $n_0$ is finite and $m_{n_0}=y_{n_0}$. Hence, there is an $n\in\N$ such that $y_n\geq u^{-1}(n)$. It thus suffices to show that for each $n_1\in\N$ 
with $y_{n_1}\geq u^{-1}(n_1)$ there is a further $n_2>n_1$ such that $y_{n_2}\geq u^{-1}(n_2)$. Since $u^{-1}$ is unbounded, there is a $k\in\N$ such that 
$u^{-1}(k)>y_{n_1}$. By (i) there is an $n>k$ such that 
\[m_n\geq u^{-1}(n)>u^{-1}(k)>y_{n_1}\,,\]
since $u^{-1}$ is also increasing. Now, choose $n_2\in\{k+1,\dotsc,n\}$ minimal such that $m_{n_2}\geq  u^{-1}(n)$. Then, $m_{n_2-1}<u^{-1}(n)$ and 
\[u^{-1}(n_2)\leq u^{-1}(n)\leq m_{n_2}=\max(m_{n_1-1},y_{n_2})=y_{n_2}\,,\]
since $m_{n_1-1}<u^{-1}(n)$.\\
\end{proof}

For a sequence $(Y_j)_{j\in\N}$ of nonnegative random variables and $n\in\N$ we define 
\begin{equation}\label{defMn'}
 M_n':=\max_{1\leq j\leq n} Y_j\,.
\end{equation}

\begin{lemma}\label{maxlemma}
 Let $(Y_i)_{i\in\N}$ be an i.i.d. sequence of nonnegative random variables and let $u:[0,\infty)\rightarrow[0,\infty)$ be an increasing and invertible function. 
\begin{enumerate}[{\normalfont (a)}]
 \item If $E[u(Y_1)]<\infty$, then $P(M_n'< u^{-1}(n) \text{ eventually })=1$.
 \item If $E[u(Y_1)]=\infty$, then $P(M_n'\geq u^{-1}(n) \text{ infinitely often })=1$.
\end{enumerate}
\end{lemma}

\begin{proof}
 We first prove (a). Since the $Y_n$ are identically distributed and also $u^{-1}$ is increasing, we have 
\begin{align*}
 \sum_{n=1}^\infty P(Y_n\geq u^{-1}(n))&=\sum_{n=1}^\infty P(Y_1\geq u^{-1}(n))\leq \int_0^\infty  P(Y_1\geq u^{-1}(x))dx\\
&= \int_0^\infty P(u(Y_1)\geq x)dx=E[u(Y_1)]<\infty\,.
\end{align*}
From the first Borel-Cantelli lemma we conclude that $P(Y_n \geq u^{-1}(n) \text{ infinitely often })=0$ and from Lemma \ref{rnlemma} we obtain 
$P(M_n'\geq u^{-1}(n) \text{ infinitely often })=0$, which is equivalent to the assertion.\\ 
Now, we turn to the proof of (b). By assumption we have 
\begin{align*}
 \sum_{n=0}^\infty P(Y_n\geq u^{-1}(n))&=\sum_{n=0}^\infty P(Y_1\geq u^{-1}(n))\geq\int_0^\infty P(Y_1\geq u^{-1}(x))dx\\
&=E[u(Y_1)]=\infty\,.
\end{align*}
By independence, the second Borel-Cantelli lemma implies that 
\[P(M_n'\geq u^{-1}(n) \text{ infinitely often })\geq P(Y_n\geq u^{-1}(n) \text{ infinitely often })=1\,.\]
\end{proof}

\begin{cor}\label{maxcor}
 Let $(Y_i)_{i\in\N}$ be an i.i.d. sequence of nonnegative random variables and let $r>0$.
\begin{enumerate}[{\normalfont (a)}]
 \item If $E[\log^+(Y_1)^r]<\infty$, then for all constants $c,L>0$ 
\[P\Bigl(\max_{1\leq i\leq\lfloor Ln^r\rfloor}Y_i< \exp\bigl(cL^{1/r}n\bigr) \text{ eventually }\Bigr)=1\,.\]
\item If $E[\log^+(Y_1)^r]=\infty$, then for every constant $c>0$ 
\[P\Bigl(M_n'\geq \exp\bigl(cn^{1/r}\bigr) \text{ infinitely often }\Bigr)=1\,.\]
\item If $E[\log^+(Y_1)^r]=\infty$, then for every constant $c>0$ and every non-decreasing sequence $(s_i)_{i\in\N}$ 
of positive integers such that $\lim_{i\to\infty}s_i=\infty$ and $\inf_{i\geq2}\frac{s_{i-1}}{s_i}>0$
\[P\Bigl(M_{s_i}'\geq \exp\bigl(c{s_i}^{1/r}\bigr) \text{ for infinitely many }i\Bigr)=1\,.\]
\end{enumerate}
\end{cor}

\begin{proof}
(a) follows from Lemma \ref{maxlemma} (a) by choosing $u(x)=(\log^+(x)/c)^r$ and noting that $M_n'< \exp\bigl(cn^{1/r}\bigr)$ eventually implies 
$\max_{1\leq i\leq\lfloor Ln^r\rfloor}Y_i< \exp\bigl(cL^{1/r}n\bigr)$ eventually. Similarly, (b) follows from Lemma \ref{maxlemma} (b). To prove (c) choose a set $G$ with $P(G)=1$ according to (b) such that for all $\om\in G$ there exists a strictly increasing sequence $(n_k)_{k\in\N}$ (depending on $\om$) with  
\[M_{n_k}'(\om)\geq \exp\bigl(\tilde{c}n_k^{1/r}\bigr) \text{for all } k\in\N\,,\]
where $\tilde{c}:=c (\inf_{i\geq2}s_{i-1}/s_i)^{-1/r}<\infty$ by the assumptions on the sequence $(s_i)_{i\in\N}$.
Then, for each $\om\in G$ and for infinitely many values of $i\in\N$ there is a $k=k_i$ such that $s_{i-1}<n_k\leq s_i$. 
The claim now follows from the chain of inequalities 
\[M_{s_i}'(\om)\geq M_{n_k}(\om)\geq \exp\Bigl(\tilde{c}n_k^{1/r}\Bigr)
\geq\exp\Bigl(s_i^{1/r}\tilde{c}\bigl(\frac{s_{i-1}}{s_i}\bigr)^{1/r}\Bigr)
\geq\exp\bigl(c s_i^{1/r}\bigr)\,.\]
\end{proof}

\begin{proof}[Proof of Lemma \ref{ml}]
For $i\in\N$ define 
\begin{equation}\label{ml1}
M_i^\star:=\max\limits_{j\in\bigcup_{k\leq i}L_i}Y_j=\max_{k\leq i} M_k\,.
\end{equation}
Note that by disjointness of the sets $L_i$ we have for the cardinality of $\bigcup_{k\leq i}L_i$:
\begin{equation}\label{ml2}
\Bigl|\bigcup_{k\leq i}L_i\Bigr|=\sum_{k=1}^i l_k\geq \sum_{k=1}^i c_2\beta^{c_3k}
=c_2\beta^{c_3}\frac{\beta^{c_3i}-1}{\beta^{c_3}-1}\geq\lceil\tilde{c}\beta^{c_3i}\rceil=:s_i\,,
\end{equation}
where  $\tilde{c}>0$ is a constant depending only on $c_2,c_3$ and $\beta$.  Hence, for each $i\in\N$, $M_i^\star$ is stochastically larger than $M_{s_i}'$ from Corollary \ref{maxcor} (c) and the integer sequence 
$(s_i)_{i\in\N}$ satisfies the above assumptions. In particular, we have 
\begin{equation}\label{ml3}
P\Bigl(M_i^\star\geq \exp\bigl(c'{s_i}^{1/r}\bigr) \text{ for infinitely many }i\Bigr)=1
\end{equation}
 for each finite constant $c'>0$. This immediately implies that 
\begin{equation}\label{ml4}
P\Bigl(M_i^\star\geq \exp\bigl(c\beta^\frac{c_3i}{r}\bigr) \text{ for infinitely many }i\Bigr)=1
\end{equation}
for each finite constant $c>0$. Now using 
\[M_i^\star=\max_{1\leq k\leq i}M_k\]
the claim follows from Lemma \ref{rnlemma} applied to the function $u(x)=r\frac{\log\log x-\log c}{c_3\log\beta}$ .\\
\end{proof}

\begin{proof}[Sketch of the proof of Lemma \ref{htlemma}]
First note that the probability $f(x,y)$ is also the probability that the continuous time random walk (CTRW)\\ $(X_t)_{t>0}=(X_t^{(1)},\dotsc,X_t^{(d)})_{t>0}$ corresponding to $p$ ever visits $y$ if it is starting at $x$. The benefit of working in continuous time here is that for CTRW the coordinates are independent, which is not true 
for discrete time random walks. Because of \eqref{drift}, letting $\tau:=\inf\{t>0\,:\,X_t^{(1)}=y_1\}$, we know that $P_x(\tau<\infty)=1$. Furthermore,
\begin{align}\label{ht1}
 f(x,y)&=P_x\bigl(\exists\, t>0\,:\,X_t=y\bigr)\geq P_x\bigl(X_\tau=y)\notag\\
 &=\int_0^\infty P_x\bigl(X_\tau=y\,\bigl|\,\tau=t\bigr)P_x\bigl(\tau \in dt\bigr)\notag\\
 &=\int_0^\infty P_x\bigl((X_t^{(2)},\dotsc,X_t^{(d)})=(y_2,\dotsc,y_d)\bigr)P_x\bigl(\tau \in dt\bigr)\notag\\
 &\geq\int_{\ga_1(y_1-x_1)}^{\ga_2(y_1-x_1)} P_x\bigl((X_t^{(2)},\dotsc,X_t^{(d)})=(y_2,\dotsc,y_d)\bigr)P_x\bigl(\tau \in dt\bigr)\,,
 \end{align}
 where $0<\ga_1<\ga_2<\infty$ are chosen such that $P_x\bigl(\ga_1(y_1-x_1)\leq\tau\leq \ga_2(y_1-x_1)\bigr)\geq1/2$. Now, since $\abs{(y_2-x_2,\dotsc,y_d-x_d)}\leq\gamma\sqrt{y_1-x_1}$, by the local CLT for continuous time random walk 
 there is a universal constant $c>0$ such that 
 \begin{equation}\label{ht2}
  P_x\bigl((X_t^{(2)},\dotsc,X_t^{(d)})=(y_2,\dotsc,y_d)\bigr)\geq\frac{c}{t^{\frac{d-1}{2}}}
 \end{equation}
 for all $t\geq \ga_1(y_1-x_1)$. Thus, from \eqref{ht1} and \eqref{ht2} we get 
 \begin{align*}
  f(x,y)&\geq\frac{c}{(\ga_2(y_1-x_1))^\frac{d-1}{2}}\int_{\ga_1(y_1-x_1)}^{\ga_2(y_1-x_1)}P_x\bigl(\tau \in dt\bigr)\\
  &\geq \frac{c}{2(\ga_2(y_1-x_1))^\frac{d-1}{2}}\,,
 \end{align*}
 yielding the claim with $c_1:=\frac{1}{2}c\ga_2^{-\frac{d-1}{2}}$.\\
\end{proof}

\bibliography{frogs}{}
\bibliographystyle{amsplain}

\end{document}